\documentclass{amsart}
\usepackage[T1]{fontenc}
\usepackage{amsfonts}
\usepackage{mathrsfs}
\usepackage{mathtools}
\usepackage{amssymb}
\usepackage[colorlinks=true,linkcolor=blue,allcolors=blue]{hyperref}
\usepackage[noabbrev,capitalize,nameinlink]{cleveref}
\usepackage{amsmath,blkarray,booktabs,bigstrut}
\usepackage{pdflscape}
\usepackage{graphicx}
\usepackage{dutchcal}
\usepackage{algorithm}
\usepackage{algpseudocode}
\usepackage[normalem]{ulem}
\usepackage{multirow} 
\usepackage{longtable} 
\usepackage{array} 
\usepackage{makecell} 
\usepackage{threeparttable}
\usepackage{booktabs}
\usepackage{comment}
\usepackage{tablefootnote}

\numberwithin{equation}{section}

\newcommand\N{{\mathbb N}}

\DeclareMathOperator{\supp}{supp}
\DeclareMathOperator{\Area}{Area}
\DeclareMathOperator{\linearspan}{span}
\renewcommand{\Re}{\text{Re}}
\renewcommand{\Im}{\text{Im}}

\theoremstyle{plain}
  \newtheorem{theorem}[subsection]{Theorem}
  \newtheorem{proposition}[subsection]{Proposition}
  
  \newtheorem{corollary}[subsection]{Corollary}
  \newtheorem{conjecture}[subsection]{Conjecture}
  \newtheorem{assumption}[subsection]{Assumption}  
  
  \theoremstyle{remark}

\theoremstyle{definition}
  
  \newtheorem{remark}[subsection]{Remark}
  \newtheorem{example}[subsection]{Example}

\begin{document}
\include{psfig}
\title[Probabilistic Approximation in RKHS]{On the Probabilistic Approximation \\ in Reproducing Kernel Hilbert Spaces}



\author{Dongwei Chen}
\address{Department of Mathematics, Colorado State University, CO, US}
\email{dongwei.chen@colostate.edu}

\author{Kai-Hsiang Wang}
\address{Department of Mathematics, Northwestern University, IL, US}
\email{kai-hsiangwang2025@u.northwestern.edu}
 
\begin{abstract}
This paper studies the probabilistic function approximation problem over reproducing kernel Hilbert spaces.
We show the existence and uniqueness of the optimizer under mild assumptions. 
Furthermore, we generalize the celebrated representer theorem to our setting, and especially when the probability measure is finitely supported, or the Hilbert space is finite-dimensional, we show that the probabilistic approximation problem turns out to be a measure quantization problem, which connects the probabilistic function approximation to the sampling theory.  Some discussions and examples are also given when the reproducing kernel Hilbert space is infinite-dimensional and the measure is infinitely supported. 

\end{abstract}


\maketitle


\section{Introduction and Problem Formulation}
Let $X$ be a set, $\mathbb{F} = \mathbb{R}$ (the set of real numbers) or $ \mathbb{C}$ (the set of complex numbers), and $\mathscr{F}(X,\mathbb{F})$ be the set of functions from $X$ to $\mathbb{F}$. 
$\mathscr{F}(X,\mathbb{F})$ is naturally equipped with the vector space structure over $\mathbb{F}$ by pointwise addition and scalar multiplication: 
\begin{equation*}
    (f+h)(x) = f(x) +h(x), \ (\beta \cdot f)(x) = \beta \cdot f(x) \;\text{for $x\in X$ and $\beta \in \mathbb{F}$}.
\end{equation*}
A vector subspace $\mathscr{H}\subset \mathscr{F}(X, \mathbb{F})$ is said to be a \textit{reproducing kernel Hilbert space} (RKHS) on $X$ if 
\begin{itemize}
    \item $\mathscr{H}$ is endowed with a Hilbert space structure $ \langle \cdot, \cdot \rangle$.
    Our convention is that it is $\mathbb{F}$-linear in the first argument. 
    \item for every $x \in X$, the linear evaluation functional $E_x: \mathscr{H} \rightarrow \mathbb{F} $, defined by $E_x(f)  =f(x)$, is bounded. 
\end{itemize}
If $\mathscr{H}$ is an RKHS on $X$, then Riesz representation theorem shows that for each $x \in X$, there exists a unique vector $k_x \in \mathscr{H}$ such that for any $f \in \mathscr{H}$, 
\begin{equation*}
    E_x(f)  = \langle f, k_x \rangle = f(x).
\end{equation*}
The function $k_x$ is called the reproducing kernel for the point $x$, and the function $K: X \times X \rightarrow \mathbb{F}$ defined by $K(y,x) = k_x(y)$ is called the \textit{reproducing kernel} for $\mathscr{H}$.
One can check that $K$ is indeed a \textit{kernel function}, meaning that for any $n\in \N$ and any $n$ distinct points $\{x_1, \cdots, x_n\} \subset X$, the matrix $(K(x_i,x_j))$ is symmetric (Hermitian) and positive semi-definite. 
It is well known that there is a one-to-one correspondence between RKHSs and kernel functions on $X$: by Moore's Theorem \cite{moore1general}, if $K: X \times X \rightarrow \mathbb{F}$ is a kernel function, then there exists a unique RKHS  $\mathscr{H}$ on $X$ such that $K$ is the reproducing kernel of $\mathscr{H}$.
We let $\mathscr{H}(K)$ denote the unique RKHS with reproducing kernel $K$, and define the \textit{feature map} $\phi: X \rightarrow \mathscr{H}(K)$ by $$\phi(x) = k_x.$$   RKHS was introduced by Zaremba \cite{zaremba1907lequation} and Mercer \cite{mercer1909xvi} and systematically studied by Aronszajn in 1950s \cite{aronszajn1950theory}.
We refer to \cite{ paulsen2016introduction, saitoh2016theory} for more details on the RKHS and its applications.

One of the interesting topics on the RKHS is interpolation. 
Let $\mathscr{H}(K)$ be an RKHS on $X$, $F= \{x_1,\cdots, x_N\}$ a finite set of distinct points in $X$, and $\{c_1, \cdots, c_N\} \subset \mathbb{F}$. 
If the matrix $(K(x_i, x_j))$ is invertible, then there exists $ f\in \mathscr{H}(K)$ such that $f(x_i) = c_i$ for all $1 \leq i \leq N$. 
However, if $(K(x_i, x_j))$ is not invertible, such $f$ may not exist. 
In this case, one is often interested in finding the best approximation in $\mathscr{H}(K)$ with the least square error:
$$\underset{f \in \mathscr{H}(K)}{\inf} \sum\limits_{i=1}^N \vert f(x_i) - c_i \vert ^2.$$
Let $\mathscr{P}(X)$ be the set of probability measures on $X$ and $\mu_N := \frac{1}{N} \sum\limits_{i=1}^N \delta_{x_i} \in \mathscr{P}(X)$.
Then the above least squares problem is equivalent to 
\begin{equation*}
    \underset{f \in \mathscr{H}(K)}{\text{inf}} \ \int_{X} \vert f(x) - g(x) \vert ^2 d\mu_N(x) = \underset{f \in \mathscr{H}(K)}{\text{inf}} \ \frac{1}{N} \sum_{i=1}^N \vert f(x_i) - c_i \vert ^2,
\end{equation*}
where $g: X \rightarrow \mathbb{F}$ is such that $g(x_i) = c_i$ for all $i=1, \cdots, N$.
This inspires us to replace $\mu_N$ with any probability measure $\mu  \in \mathscr{P}(X)$ and consider the \textit{probabilistic function approximation problem} over an RKHS.  

The general formulation is as follows. 
Throughout the paper, we assume that $X$ is a Polish space and all functions and measures are Borel measurable. By an $\mathbb{F}$-measure on $X$, we mean  a complex measure on $X$ if $\mathbb{F}=\mathbb{C}$ and a signed measure on $X$ if $\mathbb{F}=\mathbb{R}$. 
Let $g:X \rightarrow \mathbb{F}$ be a given target function, $\mu \in \mathscr{P}(X)$, and $c:\mathbb{F} \times \mathbb{F} \rightarrow \mathbb{R}^{+}$ be a non-negative cost (loss) function, where $\mathbb{R}^{+}$ is the set of nonnegative real numbers. 
Then consider the following minimization problem
\begin{equation}\label{muCostFunction}
    \underset{f \in \mathscr{H}(K)}{\text{inf}} \ \int_{X} c( f(x), g(x)) d\mu(x).
\end{equation}

If $\{x_i\}_{i=1}^N \subset X$ are independent and identically distributed samples of the probability measure $\mu$, then $\frac{1}{N}\sum\limits_{i=1}^N c( f(x_i), c_i)$ is often called the empirical risk and $\int_{X} c(f(x), g(x)) d\mu(x)$ the expected risk, where $g(x_i) = c_i$ for each $i, \ 1 \leq i \leq N$. Many efforts have been made to bound the gap between the empirical risk and the expected risk under different settings. For example, the statistical learning theory developed by Vapnik and Chervonenkis gave probabilistic bounds using the Vapnik-Chervonenkis dimension of the hypothesis function space \cite{vapnik1998statistical, vapnik2013nature, vapnik1974theory}. F. Cucker and S. Smale in \cite{cucker2002mathematical} obtained other probabilistic bounds when $c(w, z) = |w-z|^2$. 

In this work, we unify these two problems by allowing the probability measure $\mu$ to be discrete or absolutely continuous, and we consider the probabilistic function approximation problem over an RKHS. It is worth mentioning that A. D. Aydin and A. Gheondea \cite{aydin2021probability} studied function approximation in separable RKHSs using finite linear combinations of kernel functions at given points and projections onto the spans of kernels at random points.

Another motivation for this study comes from the Nevanlinna-Pick interpolation problem in complex analysis and operator theory \cite{agler2023pick, nevanlinna1919uber, pick1915beschrankungen}. Given distinct points \( z_1, z_2, \cdots, z_n \) in the open unit disk \(D = \{ z \in \mathbb{C}: |z| < 1 \}\) and target values \( t_1, t_2, \cdots, t_n \in \mathbb{C} \), the Nevanlinna-Pick interpolation problem is to find a holomorphic function \( f: D \to \mathbb{C} \) such that:
\begin{enumerate}
    \item[$(a)$] \( f(z_j) = t_j \) for all $j = 1, \cdots, n$. (Interpolation Condition)
    \item[$(b)$] for any $z \in D,  | f(z) | \leq 1$. (Pointwise Constraint)
\end{enumerate}
It is well known that the solution of Nevanlinna-Pick interpolation problem exists if and only if the \emph{Pick matrix} \( P = (\frac{1 - t_i \overline{t_j}}{1 - z_i \overline{z_j}} )_{i,j=1}^n \) is positive semi-definite \cite{pick1915beschrankungen}. Nevanlinna-Pick interpolation has been applied in many areas, such as optimal control \cite{francis1987course, tannenbaum1982modified} and signal processing \cite{blomqvist2003matrix}.

Inspired by the Nevanlinna-Pick interpolation problem, we consider the constrained version of the probabilistic approximation problem over reproducing kernel Hilbert spaces.
That is to say, for $M>0$,  consider the following problem
\begin{equation}\label{normConstraint}
    \underset{f \in \mathscr{H}(K)}{\text{inf}} \ \int_{X} c(f(x),g(x))  d\mu(x) \  \text{such that} \ \|f\|_{\mathscr{H}(K)} \leq M.
\end{equation}
Under mild conditions, the norm constraint $\| f \|_{\mathscr{H}(K)} \leq M$  leads to a pointwise constraint as in the Nevanlinna-Pick interpolation. For example, if the kernel $K$ is normalized, i.e., for any $x \in X$, $K(x, x)=\|k_x\|_{\mathscr{H}(K)}^2 =1$, then  
$$|f(x)| = |\langle f, k_x \rangle_{\mathscr{H}(K)}| \leq \| f \|_{\mathscr{H}(K)} \|k_x\|_{\mathscr{H}(K)} \leq M < +\infty;$$
and when $X$ is compact and $K: X \times X \rightarrow \mathbb{F}$ is continuous, we have
$$|f(x)| = |\langle f, k_x \rangle_{\mathscr{H}(K)}| \leq \| f \|_{\mathscr{H}(K)} \|k_x\|_{\mathscr{H}(K)} \leq M \ \underset{x \in X}{\text{max}} \sqrt{K(x, x)} < +\infty.$$

Here is another problem we consider. Instead of letting $\|f\|_{\mathscr{H}(K)} \leq M$ and controlling the norm directly, people often add an extra regularization term to the least squares to minimize (See Theorem \ref{Representer}). Especially in statistical regression and machine learning, a regularization term is preferred to perform variable selection, enhance the prediction accuracy, and prevent overfitting.
Examples of such practice are ridge regression \cite{hoerl1970ridge} and Lasso regression \cite{tibshirani1996regression}.
In this work, we consider the following minimization problem with regularization of Tikhonov type:
\begin{equation}\label{muPenalty}
    \underset{f \in \mathscr{H}(K)}{\text{inf}} \ \int_{X} c(f(x),g(x))  d\mu(x)+ \lambda \Vert f \Vert_{\mathscr{H}(K)}^p,
\end{equation}
where $\lambda>0$ is called the regularization parameter; it balances minimizing the approximation error and regulating the norm $\|f\|_{\mathscr{H}(K)}$.

We first establish the existence and uniqueness of Problem \ref{muCostFunction}, Problem \ref{normConstraint}, and Problem \ref{muPenalty} under mild assumptions in Section \ref{section: existence and uniqueness}. In Section \ref{section4}, motivated by the representer theorem stating that the optimizer minimizing the regularized loss function is in the linear span of the kernel function at the given points, we show that the optimizers of Problem \ref{normConstraint} and Problem \ref{muPenalty} have a representer-type structure, which is known as \emph{kernel mean embedding} \cite{muandet2017kernel, song2008learning}. 

Our main contribution in this work is about the representer structure in Theorem \ref{GRepresenter1} and Corollary \ref{GRepresenter2}. We show that
if $\hat{f}$ is the unique optimizer of Problem \ref{muPenalty} (See Theorem \ref{muPenaltyThm}), and the probability measure $\mu$ satisfies that for any $\mathbb{F}$-measure $\xi$ on $X$ with $\supp (\xi) \subset \supp (\mu)$, $$\int_X \|\phi(x)\|_{\mathscr{H}(K)}d|\xi|(x)<+\infty,$$
then there exists a sequence of $\mathbb{F}$-measures $\{\nu_n\}_{n=1}^{\infty}$ on $X$ with  $\supp(\nu_n) \subset \supp(\mu)$ such that 
\begin{equation*}
    \hat{f} = \underset{n \rightarrow \infty}{\text{lim}}\int_{X} \phi(x) d\nu_n(x).
\end{equation*} 
Furthermore, if $\mu$ is finitely supported, or $\mathscr{H}(K)$ is finite-dimensional, there exists an $\mathbb{F}$-measure $\nu$ on $X$ with  $\supp(\nu) \subset \supp(\mu)$ such that 
\begin{equation*}
    \hat{f} = \int_{X} \phi(x) d\nu(x),
\end{equation*} 
and Corollary \ref{GRepresenter2} shows that such $\nu$ can be finitely supported.
This connects the probabilistic function approximation to the sampling theory and indicates that the probabilistic approximation problem turns out to be a measure quantization and sampling problem. 

In Section \ref{section:conjecture}, we further conjecture that the optimizer $\hat{f}$ of Problem \ref{muPenalty} has a universal representation form even when the RKHS is infinite-dimensional and the measure is infinitely supported:
\begin{equation*}
    \hat{f} = \int_{X} \phi(x) d\xi(x),
\end{equation*} 
where $\xi$  is an $\mathbb{F}$-measure on $X$ with $\supp(\xi) \subset \supp(\mu)$. 
We also present an example in the Hardy space to support our proposal. Finally, in Section \ref{sec:conclusion}, we conclude the paper by summarizing our main results and outlining directions for future work.


\section{Existence and Uniqueness}\label{section: existence and uniqueness}
In this section, we show the existence and uniqueness of Problem \ref{muCostFunction}, Problem \ref{normConstraint}, and Problem \ref{muPenalty} under different settings. 
We first recall with the following theorem showing the existence of the optimizer to minimize the least squares $\sum\limits_{i=1}^N \vert f(x_i) - c_i \vert ^2$ and describing its representer-type structure: 
\begin{theorem}[Theorem 3.8 in \cite{paulsen2016introduction}]
    Let $\mathscr{H}(K)$ be an RKHS on $X$, $F= \{x_1,\cdots, x_N\}$ a finite set of distinct points in $X$, and $v = (c_1, \cdots, c_N)^{T} \in \mathbb{F}^{N}$. 
    Let $Q =(K(x_i, x_j))$ and $\mathscr{N}(Q)$ the null space of $Q$. 
    Then there exists $w = (\alpha_1, \cdots, \alpha_N)^T\in \mathbb{F}^N$ with $v-Qw \in \mathscr{N}(Q)$. If we let 
\begin{equation*}
            \hat{f} = \alpha_1 k_{x_1} + \cdots + \alpha_N k_{x_N}, 
\end{equation*}
then $\hat{f}$ minimizes the least squares error.
Besides, among all such minimizers in $\mathscr{H}(K)$, $\hat{f}$ is the unique function with the minimum norm. 
\end{theorem}

Our first result is about Problem \ref{muCostFunction} when the cost function $c: \mathbb{F} \times \mathbb{F} \rightarrow \mathbb{R}^+$ is given by $c(f(x), g(x)) = \vert f(x) - g(x) \vert ^p$ where  $1 \leq p<\infty$, and the feature map $\phi: X \rightarrow \mathscr{H}(K)$ is a continuous $p$-frame
for $\mathscr{H}(K)$ with respect to $\mu$; that is, there exist $0<A\leq B < +\infty $ such that for any $f \in {\mathscr{H}(K)}$, 
$$A\|f\|^p_{\mathscr{H}(K)} \leq \int_X |\langle f, \phi_x \rangle|^p d\mu(x) \leq B \|f\|^p_{\mathscr{H}(K)}.$$

\begin{theorem} \label{pexistance} 
  Let $\mathscr{H}(K)$ be an RKHS on $X$ with the feature map $\phi$.  
  Let $\mu \in \mathscr{P}(X)$ and $g \in L^p(X,\mu)$, where $1 \leq p<\infty$.
  Assume that $\phi: X \rightarrow \mathscr{H}(K)$ is a continuous $p$-frame for $\mathscr{H}(K)$ with respect to $\mu$. 
  Then the following problem 
    \begin{equation*}\label{muPNorm}
    \underset{f \in \mathscr{H}(K)}{\text{inf}} \ \int_{X} \vert f(x) - g(x) \vert ^p d\mu(x)
\end{equation*}
admits an optimizer $\hat{f} \in \mathscr{H}(K)$. Furthermore, if $p>1$, 
the optimizer is unique.
\end{theorem}

\begin{proof}
Since $g \in L^p(X,\mu)$ and $f=0 \in \mathscr{H}(K)$, we have 
 \begin{equation*}
   I:= \underset{f \in \mathscr{H}(K)}{\text{inf}} \ \int_{X} \vert f(x) - g(x) \vert ^p d\mu(x) \leq \int_{X} \vert g(x) \vert ^p d\mu(x)  < + \infty ,
\end{equation*}
i.e., the problem $I$ is bounded.
Let $\{f_i\}_{i=1}^{\infty}$ be a minimizing sequence. Then
\begin{equation*}
    \int_{X} \vert f_i(x) - g(x) \vert ^p d\mu(x) \rightarrow I.
\end{equation*}
Thus, there exists a number $0<N< \infty$ such that for each $i$ \textcolor{blue}{sufficiently large}, $ \Vert f_i - g \Vert_{L^p(X,\mu)} \leq  N$.
Then
\begin{equation*}
    \Vert f_i \Vert_{L^p(X,\mu)} \leq \Vert f_i - g \Vert_{L^p(X,\mu)} + \Vert g \Vert_{L^p(X,\mu)} \leq  N + \Vert g \Vert_{L^p(X,\mu)}, \ \text{for such $i$}.
\end{equation*}
On the other hand, since $\phi: X \rightarrow \mathscr{H}(K)$ is a continuous $p$--frame for $\mathscr{H}(K)$ with respect to $\mu$ with some lower frame bound $A>0$, then we have
\begin{equation*}
    A \Vert f_i \Vert_{\mathscr{H}(K)}^p \leq \int_{X} \vert f_i(x) \vert ^p d\mu(x) = \int_{X} \vert \langle f_i, \phi(x) \rangle_{\mathscr{H}(K)} \vert ^p d\mu(x), \;\text{for any $i$.}
\end{equation*}
Combining these results for the above sufficiently large $i$, we get
\begin{equation*}
     \Vert f_i \Vert_{\mathscr{H}(K)}^p \leq \frac{1}{A} \int_{X} \vert f_i(x) \vert ^p d\mu(x)  =  \frac{ \Vert f_i \Vert_{L^p(X,\mu)}^p}{A} \leq \frac{(N + \Vert g \Vert_{L^p(X,\mu)})^p}{A} < + \infty.
\end{equation*}
Thus  $\{f_i\}_{i=1}^{\infty}$ is a bounded sequence in $\mathscr{H}(K)$. Then $\{f_i\}_{i=1}^{\infty}$ has a weakly convergent subsequence $\{f_{i_k}\}_{k=1}^{\infty}$; i.e., there exists  $\hat{f} \in \mathscr{H}(K)$ such that for any $h \in \mathscr{H}(K)$, $  \langle f_{i_k}, h \rangle \rightarrow \langle \hat{f}, h \rangle \ \text{as} \ k \rightarrow \infty.$
\textcolor{blue}{For any $x \in X$}, taking $h = \phi(x)$, we get $f_{i_k}(x) \rightarrow \hat{f}(x)$.
Now by Fatou's lemma, we get 
\begin{equation*}
     \int_{X}  \underset{k \rightarrow \infty}{\text{lim inf}} \ \vert f_{i_k}(x) - g(x) \vert ^p d\mu(x) \leq  \underset{k \rightarrow \infty}{\text{lim inf}} \int_{X}  \vert f_{i_k}(x) - g(x) \vert ^p d\mu(x).
\end{equation*}
Using the pointwise convergence and that $\{f_i\}_{i=1}^{\infty}$ is minimizing, we obtain
\begin{equation*}
     \int_{X}  \vert \hat{f}(x) - g(x) \vert ^p d\mu(x) \leq  \underset{f \in \mathscr{H}(K)}{\text{inf}} \int_{X}  \vert f(x) - g(x) \vert ^p d\mu(x).
\end{equation*}
Therefore, $\hat{f}$ is an optimizer. 
\par
Next, we show that the optimizer is unique when $p>1$.
 Let $\hat{f_1}$ and $\hat{f_2}$ be optimizers. Then by Minkowski's inequality, we have
\begin{equation*}
    \| \frac{\hat{f_1}+\hat{f_2}}{2}-g \|_{L^p(X,\mu)} \leq \|\frac{\hat{f_1}-g}{2} \|_{L^p(X,\mu)} + \|\frac{\hat{f_2}-g}{2} \|_{L^p(X,\mu)}=I^\frac{1}{p}.
\end{equation*}
Since $I$ is the infimum and $\frac{\hat{f_1}+\hat{f_2}}{2} \in \mathscr{H}(K)$, the equality must hold and we infer the following two cases:
\begin{enumerate}
    \item $\hat{f_2}-g=0$ $\mu$-a.e.
    In this case, we have $I=0$ and $\hat{f_1}=g=\hat{f_2}$ $\mu$-a.e.

    \item $\hat{f_1}-g=\lambda(\hat{f_2}-g)$ $\mu$-a.e.\@ for some number $\lambda\geq 0$. Then $ \|\hat{f_1}-g \|_{L^p(X,\mu)} = \|\hat{f_2}-g \|_{L^p(X,\mu)} = I^{\frac{1}{p}}$ implies $\lambda=1$, and hence $\hat{f_1}=\hat{f_2}$ $\mu$-a.e. 
\end{enumerate}
In either case, we conclude $\hat{f_1}=\hat{f_2}$ $\mu$-a.e. 
\end{proof}

Note that the continuous $p$-frame condition is the same as that the $L^p$ norm is equivalent to the Hilbert space norm, as in the following inequality:
    \begin{equation*}
    C_1 \Vert f \Vert_{\mathscr{H}(K)} \leq  \Vert f \Vert_{L^p(X,\mu)}
    =\left(\int_{X} \vert \langle f, \phi(x) \rangle_{\mathscr{H}(K)} \vert ^p d\mu(x)\right)^\frac{1}{p}
    \leq  C_2 \Vert f \Vert_{\mathscr{H}(K)}, 
\end{equation*}
where $C_1 = A^{\frac{1}{p}}$ and $C_2 = B^{\frac{1}{p}}$.
Thus, we can rewrite Theorem \ref{pexistance} as the following:

\begin{corollary}\label{LpSubspace}
    Let $\mathscr{H}(K)$ be an RKHS on $X$, $\mu \in \mathscr{P}(X)$, and $g \in L^p(X,\mu)$ where $1 \leq p<\infty$.
    If $\mathscr{H}(K)$ is contained in $L^p(X,\mu)$, and the norm induced by the inner product is equivalent to the $L^p$ norm,
    then the following problem 
    \begin{equation*}
    \underset{f \in \mathscr{H}(K)}{\text{inf}} \ \int_{X} \vert f(x) - g(x) \vert ^p d\mu(x)
\end{equation*}
admits an optimizer $\hat{f} \in \mathscr{H}(K)$. Furthermore, if $p>1$, the optimizer is unique.
\end{corollary}

In the special case where $p=2$ and $\mathscr{H}(K)$ is a \textit{Hilbert} subspace of $L^2(X,\mu)$, such a unique closest vector is classically given by the orthogonal projection of $g$ onto $\mathscr{H}(K)$.
Although $L^p(X,\mu)$ is not a Hilbert space for general $p\neq 2$, our corollary (under some assumptions) provides such a unique optimizer in the probabilistic approximation sense, which can be viewed as ``projection'' onto the given subspace.

Our next result is about Problem \ref{normConstraint}. We establish the existence of the optimizer in Proposition \ref{constrain:muPenaltyThm} when the target function $g: X \to \mathbb{F}$ and the cost function $c:\mathbb{F} \times \mathbb{F} \rightarrow \mathbb{R}^{+}$ satisfy the following assumption.
\begin{assumption}\label{assumption: g and c}
    Let the target function $g:X\to \mathbb{F}$ and the cost function $c:\mathbb{F} \times \mathbb{F} \rightarrow \mathbb{R}^{+}$ be such that 
$c(0, g(\cdot)) \in L^1(X,\mu)$, and for any given $z \in \mathbb{F}$, $c(\cdot,z)$ is lower semicontinuous.
\end{assumption}

\begin{proposition}\label{constrain:muPenaltyThm}
Let $\mathscr{H}(K)$ be an RKHS on $X$, $\mu \in \mathscr{P}(X)$ and $M>0$.
Suppose the target function $g:X\to \mathbb{F}$ and the cost $c:\mathbb{F} \times \mathbb{F} \rightarrow \mathbb{R}^{+}$ satisfy Assumption \ref{assumption: g and c}.
Then the following problem
\begin{equation*}
    \underset{f \in \mathscr{H}(K)}{\text{inf}} \ \int_{X} c(f(x),g(x))  d\mu(x) \  \text{such that} \ \|f\|_{\mathscr{H}(K)} \leq M
\end{equation*}
admits an optimizer $\hat{f} \in \mathscr{H}(K)$. 
Furthermore, 
when $c(f(x),g(x)) = |f(x)-g(x)|^p$ with $p>1$, the optimizer is unique. 
\end{proposition}

The proof of Proposition \ref{constrain:muPenaltyThm} follows from the argument of Theorem \ref{pexistance}. The key idea is to show that the minimizing sequence $\{f_i\}_{i=1}^{\infty}$ is bounded and thus has a weakly convergent subsequence. Then the weak limit is an optimizer. However, we do not need the continuous $p$-frame condition in Theorem \ref{pexistance}, since we already have the norm constraint: for each $i$, $\|f_i\|_{\mathscr{H}(K)} \leq M$.

\begin{proof}
Since  $c(0, g(\cdot)) \in L^1(X,\mu)$ and $f=0 \in \mathscr{H}(K)$ with $\|f\|_{\mathscr{H}(K)} \leq M$, then 
 \begin{equation*}
   I_c:= \underset{f \in \mathscr{H}(K)}{\text{inf}} \ \int_{X} c( f(x), g(x))  d\mu(x) \leq \int_{X} c(0, g(x)) d\mu(x)  < + \infty.
\end{equation*}
Hence, the problem $I_c$ is bounded. Let $\{f_i\}_{i=1}^{\infty}$ be a minimizing sequence such that for each $i$, $\|f_i\|_{\mathscr{H}(K)} \leq M$ and
\begin{equation*}
     \int_{X} c( f_i(x), g(x))  d\mu(x)  \rightarrow I_c < + \infty.
\end{equation*}
Since $\{f_i\}_{i=1}^{\infty}$ is a bounded sequence in $\mathscr{H}(K)$, $\{f_i\}_{i=1}^{\infty}$ has a weakly convergent subsequence $\{f_{i_k}\}_{k=1}^{\infty}$ with weak limit $\hat{f} \in \mathscr{H}(K)$. Then for each $x \in X$, $f_{i_k}(x)= \langle f_{i_k}, \phi(x) \rangle \rightarrow \hat{f}(x)= \langle \hat{f}, \phi(x) \rangle $.
By the lower semicontinuity of the cost function $c$ and Fatou's lemma, we get 
\begin{equation*}
\begin{split}
    \int_{X}  c( \hat{f}(x), g(x)) d\mu(x) &\leq \int_{X}  \underset{k \rightarrow \infty}{\text{lim inf}} \ c( \hat{f}_{i_k}(x), g(x)) d\mu(x) \\
    &\leq  \underset{k \rightarrow \infty}{\text{lim inf}} \int_{X}  c( \hat{f}_{i_k}(x), g(x)) d\mu(x) =I_c,
\end{split}
\end{equation*}
where the last equality follows from that $\{f_i\}_{i=1}^{\infty}$ is a minimizing sequence for $I_c$.
On the other hand, since $f_{i_k}$ converges to $\hat{f}$ weakly in $\mathscr{H}(K)$, we have $\Vert \hat{f} \Vert_{\mathscr{H}(K)} \leq  \underset{k \rightarrow \infty}{\text{lim inf}}  \ \Vert f_{i_k} \Vert_{\mathscr{H}(K)} \leq M.$
Therefore, $\hat{f}$ is an optimizer.

Let us show the uniqueness when $c(f(x),g(x)) = |f(x)-g(x)|^p$ with $p>1$. 
Suppose $\hat{f_1}$ and $\hat{f_2}$ are optimizers attaining $I_c$. 
Note that $\frac{\hat{f_1}+\hat{f_2}}{2} \in \mathscr{H}(K)$ and $\|\frac{\hat{f_1}+\hat{f_2}}{2}\|_{\mathscr{H}(K)} \leq M$.
Since $c(\cdot,z) = |\cdot - z|^p$ is convex for any given $z$, then
\begin{equation*}
\begin{split}
     \int_{X}&  |\frac{\hat{f_1}(x)+\hat{f_2}(x)}{2}- g(x)|^p d\mu(x)
     \leq \frac{\int_{X} |\hat{f_1}(x)- g(x)|^p d\mu(x)}{2} +  \frac{\int_{X}  |\hat{f_2}(x)- g(x)|^p d\mu(x)}{2} =I_c. 
\end{split}
\end{equation*}
 Thus, $\frac{\hat{f_1}+\hat{f_2}}{2}$ is also an optimizer. Then the equality in the Minkowski inequality about $\| \frac{\hat{f_1}+\hat{f_2}}{2}-g \|_{L^p(X,\mu)}$ must hold, and using the same argument from the uniqueness proof of Theorem \ref{pexistance}, we infer $\hat{f_1}=\hat{f_2}$ $\mu$-a.e.
\qedhere
\end{proof}

Our last result in this section involves Problem \ref{muPenalty} by adding an extra regularization term of Tikhonov type to the original minimization problem:
\begin{equation*}
    \underset{f \in \mathscr{H}(K)}{\text{inf}} \ \int_{X} c(f(x),g(x))  d\mu(x)+ \lambda \Vert f \Vert_{\mathscr{H}(K)}^p,
\end{equation*}
where $p>0$ and $\lambda>0$. We show the existence and uniqueness of the optimizer in the following theorem. The key idea is also to show that the minimizing sequence $\{f_i\}_{i=1}^{\infty}$ is bounded, which is guaranteed by the regularization.  
\begin{theorem}\label{muPenaltyThm}
Let $\mathscr{H}(K)$ be an RKHS on $X$, $\mu \in \mathscr{P}(X)$ and $0 < p <\infty$. Suppose the target function $g:X\to \mathbb{F}$ and the cost $c:\mathbb{F} \times \mathbb{F} \rightarrow \mathbb{R}^{+}$ satisfy Assumption \ref{assumption: g and c}.
Then Problem \ref{muPenalty}
admits an optimizer $\hat{f} \in \mathscr{H}(K)$. 
Furthermore,  when $p> 1$ and $c(\cdot,z)$ is convex for any given $z \in \mathbb{F}$, the optimizer is unique. 
\end{theorem}
\begin{proof}
Since  $c(0, g(\cdot)) \in L^1(X,\mu)$ and $f=0 \in \mathscr{H}(K)$, then 
 \begin{equation*}
   I_g:= \underset{f \in \mathscr{H}(K)}{\text{inf}} \ \int_{X} c( f(x), g(x))  d\mu(x) + \lambda \Vert f \Vert^p_{\mathscr{H}(K)} \leq \int_{X} c(0, g(x)) d\mu(x)  < + \infty.
\end{equation*}
Hence, the problem $I_g$ is bounded. Let $\{f_i\}_{i=1}^{\infty}$ be a minimizing sequence. Then
\begin{equation*}
     \int_{X} c( f_i(x), g(x))  d\mu(x) + \lambda \Vert f_i \Vert^p_{\mathscr{H}(K)} \rightarrow I_g < + \infty.
\end{equation*}
Then  there exists $0<T<\infty$ such that for each $i$ \textcolor{blue}{sufficiently large}, 
\begin{equation*}
    \int_{X} c( f_i(x), g(x))  d\mu(x) + \lambda \Vert f_i \Vert^p_{\mathscr{H}(K)}  \leq  T.
\end{equation*}
Thus  $\{f_i\}_{i=1}^{\infty}$ is a bounded sequence in $\mathscr{H}(K)$. 
Using the same argument from the proof of Proposition \ref{constrain:muPenaltyThm}, we have that $\{f_i\}_{i=1}^{\infty}$ has a weakly convergent subsequence $\{f_{i_k}\}_{k=1}^{\infty}$ with weak limit $\hat{f} \in \mathscr{H}(K)$ and
\begin{equation*}
    \int_{X}  c( \hat{f}(x), g(x)) d\mu(x) \leq  \underset{k \rightarrow \infty}{\text{lim inf}} \int_{X}  c( \hat{f}_{i_k}(x), g(x)) d\mu(x).
\end{equation*}
On the other hand, since $f_{i_k}$ converges to $\hat{f}$ weakly in $\mathscr{H}(K)$, we have $\Vert \hat{f} \Vert_{\mathscr{H}(K)}^p \leq  \underset{k \rightarrow \infty}{\text{lim inf}}  \ \Vert f_{i_k} \Vert_{\mathscr{H}(K)}^p.$
Furthermore, by the superadditivity of the limit inferior, we get 
\begin{equation*}
\begin{split}
      \underset{k \rightarrow \infty}{\text{lim inf}} \int_{X}  c( \hat{f}_{i_k}(x), g(x)) &d\mu(x) + \lambda \ \underset{k \rightarrow \infty}{\text{lim inf}}  \ \Vert f_{i_k} \Vert^p_{\mathscr{H}(K)}  \\
      &\leq \underset{k \rightarrow \infty}{\text{lim inf}} \int_{X} c( f_{i_k}(x), g(x))  d\mu(x) + \lambda \Vert f_{i_k} \Vert^p_{\mathscr{H}(K)}=I_g,
\end{split}
\end{equation*}
where the last equality follows from that $\{f_i\}_{i=1}^{\infty}$ is minimizing. 
Combining the above results, we obtain
\begin{equation*}
     \int_{X}  c (\hat{f}(x), g(x)) d\mu(x) + \lambda \Vert \hat{f} \Vert^p_{\mathscr{H}(K)} \leq  I_g.
\end{equation*}
Therefore, $\hat{f}$ is an optimizer.
\par

Next, we show the uniqueness when $p > 1$ and $c(\cdot,z)$ is convex for any fixed $z \in \mathbb{F}$.
Let $\hat{f_1}$ and $\hat{f_2}$ be optimizers attaining $I_g$. 
Since $\frac{\hat{f_1}+\hat{f_2}}{2} \in \mathscr{H}(K)$, we have 
\begin{equation*}
\begin{split}
     \int_{X}&  c (\frac{\hat{f_1}(x)+\hat{f_2}(x)}{2}, g(x)) d\mu(x)  +  \lambda \Vert \frac{\hat{f_1}+\hat{f_2}}{2} \Vert^p_{\mathscr{H}(K)}
     \geq I_g\\
     &=\frac{\lambda \Vert \hat{f_1} \Vert^p_{\mathscr{H}(K)}}{2} +  \frac{\lambda \Vert \hat{f_2} \Vert^p_{\mathscr{H}(K)}}{2} +  \frac{\int_{X} c (\hat{f_1}(x), g(x)) d\mu(x)}{2} +  \frac{\int_{X} c (\hat{f_2}(x), g(x)) d\mu(x)}{2}. 
\end{split}
\end{equation*}
Since $c(\cdot,z)$ is convex for any given $z$, we then have
\begin{equation*}
     \Vert \frac{\hat{f_1}+\hat{f_2}}{2} \Vert^p_{\mathscr{H}(K)}
     \geq \frac{\Vert \hat{f_1} \Vert^p_{\mathscr{H}(K)}}{2} +  \frac{\Vert \hat{f_2} \Vert^p_{\mathscr{H}(K)}}{2}.
\end{equation*}
On the other hand, by triangle inequality and that the map $x \mapsto |x|^p$ is (strictly) convex for $p> 1$, we get
\begin{equation*}
    \left\Vert \frac{\hat{f_1}+\hat{f_2}}{2} \right \Vert^p_{\mathscr{H}(K)}
    \leq \left(\frac{\Vert \hat{f_1} \Vert_{\mathscr{H}(K)}}{2} +  \frac{\Vert \hat{f_2} \Vert_{\mathscr{H}(K)}}{2}\right)^p
    \leq \frac{\Vert \hat{f_1} \Vert^p_{\mathscr{H}(K)}}{2} +  \frac{\Vert \hat{f_2} \Vert^p_{\mathscr{H}(K)}}{2}.
\end{equation*}
Hence we see that the equalities above must hold, and we infer $\hat{f_1}=c\hat{f_2}$ for some $c\geq 0$ as well as $\Vert \hat{f_1} \Vert_{\mathscr{H}(K)}=\Vert \hat{f_2} \Vert_{\mathscr{H}(K)}$ 
(the case $\hat{f_1}=0$ implies $\hat{f_2}=0$, and vice versa).
Therefore $c=1$ and  $\hat{f_1}=\hat{f_2}$.
\qedhere
\end{proof}


\section{Probabilistic Representer Theorem}\label{section4}

As witnessed by the work of Wahba \cite{wahba1990spline} and followed by Schölkopf, Herbrich, and Smola \cite{scholkopf2001generalized}, the celebrated representer theorem (Theorem \ref{Representer} below) states that the optimizer in an RKHS that minimizes the regularized loss function is in the linear span of the kernel function at the given points.
\begin{theorem}[Theorem 8.7 and 8.8 in \cite{paulsen2016introduction}]\label{Representer}
     Let $\mathscr{H}(K)$ be an RKHS on $X$,
    $F= \{x_1,\cdots, x_N\}$ a finite set of distinct points in $X$, and $\{c_1, \cdots, c_N\} \subset \mathbb{F}$. Let $L: \mathbb{F}^N \rightarrow \mathbb{R}^+$ be a convex function representing the loss and consider the minimizing problem 
    $$\underset{f \in \mathscr{H}(K)}{\inf} L(f(x_1), \cdots, f(x_N)) + \Vert f \Vert_{\mathscr{H}(K)}^2.$$
    Then the optimizer to this problem exists and is unique.
    Furthermore, the optimizer is in the linear span of the functions $k_{x_1}, \cdots, k_{x_N}$.
\end{theorem}

The representer theorem is useful in practice since it turns the minimization problem into a finite-dimensional optimization problem. Recent advances have extended this framework into broader settings, including vector-valued RKHS and other regularization terms that are not Tikhonov-type. For example, M. Belkin, P. Niyogi, and V. Sindhwani in \cite{belkin2006manifold} proposed a representer theorem by adding the manifold regularization term in the minimization problem. C. A. Michelli and M. Pontil
 in \cite{micchelli2005learning} gave a representation theorem for reproducing kernel Hilbert spaces of $\mathbb{R}^n$-valued functions, which was then used for the multi-task learning \cite{evgeniou2005learning}. 
H. Kadri et al. in \cite{kadri2016operator} proved an analog of the representer theorem for the function-valued RKHS (a.k.a RKHS of operators) under the functional regression problem setting. 
H. Q. Minh, L. Bazzani, and V. Murino in \cite{minh2016unifying} showed a representer theorem for a minimization problem in function-valued RKHSs, which
includes vector-valued manifold regularization and multi-view learning as special cases, and A. Gheondea and C. Tilki in \cite{gheondea2024localisation} further proved some representer theorems for a localized version of the minimizing problem in \cite{minh2016unifying}.

The fact that the optimizer $\hat{f}$ is in the linear span of the functions $k_{x_1}, \cdots, k_{x_N}$ implies that there exist $\omega_1, \cdots, \omega_N \in \mathbb{F}$ such that 
$$\hat{f} = \sum_{i=1}^N \omega_i k_{x_i} = \sum_{i=1}^N \omega_i \phi(x_i), $$
where $\phi:X \rightarrow \mathscr{H}(K)$ is the feature map. Using the idea of transforming the sum into an integration with respect to some $\mathbb{F}$-measure $\nu_f:= \sum\limits_{i=1}^N \omega_i \delta_{x_i}$, we have
$$\hat{f} = \int_{X} \phi(x) d\nu_f(x), $$
which indicates that the optimizer has a measure representation in the representer theorem.
Therefore, we would like to know whether there is a corresponding version of the representer theorem under the probabilistic approximation setting that we introduced in previous sections: if $\hat{f} \in \mathscr{H}(K)$ is the optimizer in Theorem \ref{muPenaltyThm}, there may exist an $\mathbb{F}$-measure $\nu$ with $\supp(\nu) \subset \supp(\mu)$ such that 
\begin{equation*}
    \hat{f} = \int_{X} \phi(x) d\nu(x).
\end{equation*} 
However, the above integral is RKHS-valued, and we need to make it well-defined. Indeed, if $\nu$ is a probability measure on $X$ and $\mathscr{H}(K)$ is separable, the integral $ \int_{X} \phi(x) d\nu(x)$ is well-defined with mild assumptions:  Proposition 4 in \cite{aydin2021probability} showed that if $\int_X \|\phi(x)\|_{\mathscr{H}(K)}^2d\nu(x)<\infty  $, then $ \int_{X} \phi(x) d\nu(x)$ is a Bochner integral, and Theorem 3.11 in \cite{kukush2020gaussian} showed that $ \int_{X} \phi(x) d\nu(x)$ is a Bochner integral if and only if $\int_X \|\phi(x)\|_{\mathscr{H}(K)}d\nu(x)<\infty $. These conditions naturally hold if the kernel $K: X \times X \rightarrow \mathbb{F}$ is normalized ($\|\phi(x)\|_{\mathscr{H}(K)} =1$).  
Moreover, $\int_{X} \phi(x) d\nu(x)$ is called the \emph{kernel mean embedding} of $\nu$ in $\mathscr{H}(K)$, which leads to the \emph{maximum mean discrepancy} (MMD) metric on the set of probability measures on $X$ if the kernel $K$ is characteristic: 
$$\text{MMD}(\nu_1, \nu_2) = \left\| \int_X \phi(x) d\nu_1(x) -\int_X \phi(x) d\nu_2(x) \right\|_{\mathscr{H}(K)},$$
where $\nu_1$ and $\nu_2$ are probability measures on $X$. See    \cite{muandet2017kernel, song2008learning} for more details. 

In this work, since $\nu$ is an $\mathbb{F}$-measure, the integral $ \int_{X} \phi(x) d\nu(x)$ is defined as the Pettis integral: for any $f \in \mathscr{H}(K)$,  
     \begin{align*}
            \langle \int_{X} \phi(x) d\nu(x), f \rangle_{\mathscr{H}(K)} 
    =\int_{X} \langle \phi(x), f \rangle_{\mathscr{H}(K)} d\nu(x). 
    \end{align*}
By Cauchy–Schwarz Inequality, the above integral is well-defined as long as 
     $$\int_X \|\phi(x)\|_{\mathscr{H}(K)} d|\nu|(x) < +\infty, $$
where $\vert \nu \vert$ is the variation measure of $\nu$.
We confirm our speculation in the following representer theorem by showing that the optimizer is the limit of a sequence of kernel embeddings of some $\mathbb{F}$-measures.

\begin{theorem}[Probabilistic Representer]\label{GRepresenter1}
 Let $\hat{f}$ be the unique optimizer in Theorem \ref{muPenaltyThm}. 
 Assume the probability measure $\mu$ satisfies that for any $\mathbb{F}$-measure $\xi$ on $X$ with $\supp (\xi) \subset \supp(\mu)$, the following holds: 
\begin{equation}\label{cond:bddness_representer_thm}
    \int_X \|\phi(x)\|_{\mathscr{H}(K)}d|\xi|(x)<+\infty.
\end{equation}
Then there exists a sequence of $\mathbb{F}$-measures $\{\nu_n\}_{n=1}^{\infty}$ on $X$ with  $\supp(\nu_n) \subset \supp(\mu)$ such that 
\begin{equation*}
    \hat{f} = \underset{n \rightarrow \infty}{\text{lim}}\int_{X} \phi(x) d\nu_n(x).
\end{equation*} 
Furthermore, if $\mu$ is finitely supported, or $\mathscr{H}(K)$ is finite-dimensional, then there exists an $\mathbb{F}$-measure $\nu$ on $X$ with  $\supp(\nu) \subset \supp(\mu)$ such that 
\begin{equation*}
    \hat{f} = \int_{X} \phi(x) d\nu(x).
\end{equation*} 
\end{theorem}

The finiteness condition \ref{cond:bddness_representer_thm} in Theorem \ref{GRepresenter1} holds when $\mu$ is finitely supported, or when $X$ is compact and the kernel $K: X \times X \rightarrow \mathbb{F}$ is continuous. The condition \ref{cond:bddness_representer_thm} also holds when $K$ is normalized: for any $x \in X$, $K(x, x) = \|\phi(x)\|_{\mathscr{H}(K)}^2 =1$.
Furthermore, the condition \ref{cond:bddness_representer_thm} holds on the Hardy space $H^2(D)$ when $\supp \mu$ is a compact subset of the open unit disk $D$, as we will see in Example \ref{ex:hardy} later.

\begin{proof}
Define
    \begin{equation*}
        \mathscr{A}:=\left\{\int_{X} \phi(x) d\xi(x): \xi \ \text{is an $\mathbb{F}$-measure on $X$}, \supp(\xi) \subset \supp(\mu) \right \},
    \end{equation*}
    and let $\Omega$ be the closure of $\mathscr{A}$ in $\mathscr{H}(K)$. Here, the RKHS-valued integral above is defined as the Pettis integral: for any $f \in \mathscr{H}(K)$,  
     \begin{align*}
            \langle \int_{X} \phi(x) d\xi(x), f \rangle_{\mathscr{H}(K)} 
    =\int_{X} \langle \phi(x), f \rangle_{\mathscr{H}(K)} d\xi(x). 
    \end{align*}
    By the Assumption \ref{cond:bddness_representer_thm}, the above integral is well-defined, since
    \begin{align*}
            \left|\langle \int_{X} \phi(x) d\xi(x), f \rangle_{\mathscr{H}(K)} \right|
    &=\left|\int_{X} \langle \phi(x), f \rangle_{\mathscr{H}(K)} d\xi(x)\right| \\
    &\leq \|f\|_{\mathscr{H}(K)}\int_{X} \|\phi(x)\|_{\mathscr{H}(K)} d|\xi|(x) < +\infty,
    \end{align*}
where $\vert \xi \vert$ is the variation measure of $\xi$. 
It is easy to see that $\Omega$ is a closed subspace of $ \mathscr{H}(K)$. Let $P_\Omega$ be the orthogonal projection of $ \mathscr{H}(K)$ onto $\Omega$. 
Note that for any $x \in \supp(\mu)$, $\phi(x) \in \Omega$ by taking $\xi$ as the delta measure at $x$.  Therefore, for any $x \in \supp(\mu)$, the optimizer $\hat{f}$ satisfies
\begin{equation*}
    (P_\Omega \hat{f})(x) = \langle P_\Omega \hat{f}, \phi(x) \rangle_{\mathscr{H}(K)} = \langle  \hat{f}, P_\Omega \phi(x) \rangle_{\mathscr{H}(K)} = \langle  \hat{f},  \phi(x) \rangle_{\mathscr{H}(K)} =  \hat{f}(x).
\end{equation*}
Therefore, 
  \begin{equation*}
   \int_{X} c( P_\Omega \hat{f}(x), g(x))  d\mu(x) + \lambda \Vert P_\Omega \hat{f} \Vert^p_{\mathscr{H}(K)} \leq  \int_{X} c( \hat{f}(x), g(x))  d\mu(x) + \lambda  \Vert \hat{f} \Vert^p_{\mathscr{H}(K)}.
\end{equation*}
Hence, $P_\Omega \hat{f}$ is also an optimizer.
Since the optimizer is unique,  we conclude $P_\Omega \hat{f}= \hat{f}$ and thus $\hat{f} \in \Omega$. Therefore, there exists a sequence of $\mathbb{F}$-measures $\{\nu_n\}_{n=1}^{\infty}$ on $X$ with  $\supp(\nu_n) \subset \supp(\mu)$ for each $n$, such that 
\begin{equation*}
    \hat{f} = \underset{n \rightarrow \infty}{\text{lim}}\int_{X} \phi(x) d\nu_n(x).
\end{equation*} 
If $\mu$ is finitely supported or $\mathscr{H}(K)$ is finite-dimensional, then the set $\mathscr{A}$ is automatically closed and $\Omega = \mathscr{A}$.
Thus there exists an $\mathbb{F}$-measure $\nu$ on $X$ with  $\supp(\nu) \subset \supp(\mu)$ such that 
\begin{equation*}
    \hat{f} = \int_{X} \phi(x) d\nu(x). \qedhere
\end{equation*} 
\end{proof}

We can furthermore assume that the $\mathbb{F}$-measures $\xi$ on $X$ are finitely supported, and the set $\mathscr{A}$ then becomes the linear span of $\{k_x\}_{x \in \supp(\mu)}$.
This leads to the following corollary:
\begin{corollary}[Discrete Probabilistic Representer]\label{GRepresenter2}
    Let $\hat{f}$ be the unique optimizer in Theorem \ref{muPenaltyThm}.
    Then $\hat{f}$ is in the closure of the linear span of $\{k_x\}_{x \in \supp(\mu)}$.
    Furthermore, if $\mu$ is finitely supported, or $\mathscr{H}(K)$ is finite-dimensional, 
 then $\hat{f}$ is in the linear span of $\{k_x\}_{x \in \supp(\mu)}$.
\end{corollary}

 \begin{proof}
Here, we use the following definition of $\mathscr{A}$:
     \begin{equation*}
        \mathscr{A}:= \text{span}{\{k_x\}_{x \in \supp(\mu)}} = \left\{\sum_{i=1}^m w_i k_{x_i}: m \in \mathbb{N}^{+},  \{w_i\}_{i=1}^m \subset \mathbb{F},  \{x_i\}_{i=1}^m \subset \supp(\mu) \right \},
    \end{equation*}
and let $\Omega$ be the closure of $\mathscr{A}$ in $\mathscr{H}(K)$.
Then $\Omega$ is a closed linear subspace of $ \mathscr{H}(K)$. 
Using the same arguments as in Theorem \ref{GRepresenter1}, we conclude $\hat{f} \in \Omega$.
If $\mu$ is finitely supported or $\mathscr{H}(K)$ is finite-dimensional, the set $\mathscr{A}$ is automatically closed and thus $\hat{f} \in \Omega = \mathscr{A}$.
\end{proof}

Note that when $\mu$ is finitely supported, our result in  Corollary \ref{GRepresenter2} recovers the classical representer theorem (Theorem \ref{Representer}) while extending it to broader settings.
On the other hand, when $\mathscr{H}(K)$ is finite-dimensional,  the unique optimizer is also in the linear span of $\{k_x\}_{x \in \supp(\mu)}$, which does not depend on the cardinality of the support of the measure $\mu$. There are many finite-dimensional reproducing kernel Hilbert spaces in applications. For example, the vector space of polynomials of degree $n-1$ or less is an $n$-dimensional RKHS (See page 7 of \cite{wahba1990spline}), and the associated graph-based RKHS of a given connected graph is finite-dimensional (See \cite{seto2014gram} and page 56 of \cite{saitoh2016theory}). Furthermore, as a subspace of Sobolev space, the space $\mathcal{S}(z_1, \cdots, z_k)$ of polynomial splines of order $r$ with simple knot sequence $z_1 < \cdots < z_k$  is a finite-dimensional RKHS with dimension $r+k$ (See page 110 of \cite{berlinet2011reproducing}), which has been used in many areas, such as nonparametric regression in statistics \cite{berlinet2011reproducing, wahba1990spline} and finite element method for partial differential equations \cite{hollig2003finite}. 
Corollary \ref{GRepresenter2} indicates that the probabilistic approximation problem over a finite-dimensional RKHS turns out to be a measure quantization and sampling problem about the measure $\mu$ with respect to the regularized loss function in Problem \ref{muPenalty}, which connects the probabilistic approximation problem to the sampling theory.

\begin{remark}
     When the cost function $c: \mathbb{F} \times \mathbb{F} \rightarrow \mathbb{R}^+$ and the target function $g: X \rightarrow \mathbb{F}$ satisfy the corresponding assumptions in Theorem \ref{muPenaltyThm}, the existence and uniqueness still hold for the following more general problem
\begin{equation}
    \underset{f \in \mathscr{H}(K)}{\text{inf}} \ \int_{X} c(f(x),g(x))  d\mu(x)+ h( \Vert f \Vert_{\mathscr{H}(K)}),
\end{equation}
where $h:\mathbb{R}^{+} \rightarrow \mathbb{R}^{+}$ can be any strictly convex function.
Furthermore, when $h:\mathbb{R}^{+} \rightarrow \mathbb{R}^{+}$ is increasing and strictly convex, the probabilistic representer theorems (Theorem \ref{GRepresenter1} and Corollary \ref{GRepresenter2}) also hold. 
\end{remark}

This section is ended with the following corollaries showing the representer theorems for the optimizer of Problem \ref{normConstraint} under the condition \ref{cond:bddness_representer_thm}.
\begin{corollary}\label{constraint: GRepresenter1}
  Let $\hat{f}$ be the unique optimizer in Proposition \ref{constrain:muPenaltyThm} when the cost function $c(f(x), g(x)) = |f(x)-g(x)|^p$ where $p>1$.
 Assume the probability measure $\mu$ satisfies that for any $\mathbb{F}$-measure $\xi$ on $X$ with $\supp (\xi) \subset \supp(\mu)$,
\begin{equation*}
    \int_X \|\phi(x)\|_{\mathscr{H}(K)}d|\xi|(x)<+\infty.
\end{equation*}
Then there exists a sequence of $\mathbb{F}$-measures $\{\nu_n\}_{n=1}^{\infty}$ on $X$ with  $\supp(\nu_n) \subset \supp(\mu)$ such that 
\begin{equation*}
    \hat{f} = \underset{n \rightarrow \infty}{\text{lim}}\int_{X} \phi(x) d\nu_n(x).
\end{equation*} 
Furthermore, if $\mu$ is finitely supported, or $\mathscr{H}(K)$ is finite-dimensional, then there exists an $\mathbb{F}$-measure $\nu$ on $X$ with  $\supp(\nu) \subset \supp(\mu)$ such that 
\begin{equation*}
    \hat{f} = \int_{X} \phi(x) d\nu(x).
\end{equation*} 
\end{corollary}

\begin{proof}
Define
    \begin{equation*}
        \mathscr{A}:=\left\{\int_{X} \phi(x) d\xi(x): \xi \ \text{is an $\mathbb{F}$-measure on $X$}, \supp(\xi) \subset \supp(\mu) \right \},
    \end{equation*}
    and let $\Omega$ be the closure of $\mathscr{A}$ in $\mathscr{H}(K)$, where the integral within $\mathscr{A}$ is also defined as the Pettis integral.
Then $\Omega$ is a closed subspace of $ \mathscr{H}(K)$.  Let $P_\Omega$ be the orthogonal projection of $ \mathscr{H}(K)$ onto $\Omega$. Thus, for any $x \in \supp(\mu), (P_\Omega \hat{f})(x) =  \hat{f}(x)$ (using the same argument in Theorem \ref{GRepresenter1}). Therefore, 
  \begin{equation*}
   \int_{X} c( P_\Omega \hat{f}(x), g(x))  d\mu(x) =  \int_{X} c( \hat{f}(x), g(x))  d\mu(x).
\end{equation*}
Since $\|P_\Omega \hat{f}\|_{\mathscr{H}(K)} \leq \| \hat{f}\|_{\mathscr{H}(K)} \leq M$, then  $P_\Omega \hat{f} \in \Omega$ is also an optimizer.
Since the optimizer is unique,  then $\hat{f} = P_\Omega \hat{f} \in \Omega$.
If $\mu$ is finitely supported or $\mathscr{H}(K)$ is finite-dimensional, then the set $\mathscr{A}$ is automatically closed and thus $ \hat{f} \in \Omega = \mathscr{A}$.
\end{proof}

\begin{corollary}\label{constraint:GRepresenter2}
    Let $\hat{f}$ be the unique optimizer in Proposition \ref{constrain:muPenaltyThm} when the cost function $c(f(x), g(x)) = |f(x)-g(x)|^p$ where $p>1$.
    Then $\hat{f}$ is in the closure of the linear span of $\{k_x\}_{x \in \supp(\mu)}$.
    Furthermore, if $\mu$ is finitely supported, or $\mathscr{H}(K)$ is finite-dimensional, 
 then $\hat{f}$ is in the linear span of $\{k_x\}_{x \in \supp(\mu)}$.
\end{corollary}
 \begin{proof}
Here, we use the same notation of $\mathscr{A}$:
     \begin{equation*}
        \mathscr{A}:= \linearspan{\{k_x\}_{x \in \supp(\mu)}} = \left\{\sum_{i=1}^m w_i k_{x_i}: m \in \mathbb{N}^{+},  \{w_i\}_{i=1}^m \subset \mathbb{F},  \{x_i\}_{i=1}^m \subset \supp(\mu) \right \},
    \end{equation*}
and let $\Omega$ be the closure of $\mathscr{A}$. Using the same argument from Corollary \ref{constraint: GRepresenter1}, we conclude that $\hat{f} \in \Omega$. 
 If $\mu$ is finitely supported or $\mathscr{H}(K)$ is finite-dimensional, the set $\mathscr{A}$ is automatically closed and thus $\hat{f} \in \Omega = \mathscr{A}$.
\end{proof}


 \vspace{0.3cm}

\section{Discussions on the Representer Theorem}\label{section:conjecture}
The preferable result in the probabilistic representer theorem (Theorem \ref{GRepresenter1}) is that the minimizer $\hat{f}$ is the kernel embedding of some $\mathbb{F}$-measure $\nu$ such that $\hat{f}$ can be represented directly by the $\mathbb{F}$-measure $\nu$, instead of an approximating sequence $\{\nu_n\}$. 
We conjecture that such a desirable closed form holds merely under the assumption \ref{cond:bddness_representer_thm}:

\begin{conjecture}\label{conj:representer}
    Let $\hat{f}$ be the unique optimizer in Theorem \ref{muPenaltyThm}.
    Assume the probability measure $\mu$ satisfies that for any $\mathbb{F}$-measure $\xi$ on $X$ with $\supp (\xi) \subset \supp(\mu)$, the following holds: 
\begin{equation*}
    \int_X \|\phi(x)\|_{\mathscr{H}(K)}d|\xi|(x)<+\infty.
\end{equation*}
Then there exists an $\mathbb{F}$-measure $\nu$ on $X$ with  $\supp(\nu) \subset \supp(\mu)$ such that 
\begin{equation*}
    \hat{f} = \int_{X} \phi(x) d\nu(x).
\end{equation*} 
\end{conjecture}

Once this conjecture is true, the optimizer of Problem \ref{muPenalty} has a measure representation form, which makes the probabilistic approximation problem become a sampling problem from the perspective of numerical applications: sample the probability measure $\mu$ to get an $\mathbb{F}$-measure $\nu = \sum\limits_{i=1}^m w_i \delta_{y_i}$ on $X$ and a candidate $\int_{X} \phi(x) d\nu(x)$ for the optimizer, and then update $\nu$ to make the regularized loss function in Problem \ref{muPenalty} as small as possible. Therefore, this conjecture has potential applications in areas where function approximation over RKHSs plays a central role. 
To support this conjecture as well as to illustrate Theorem \ref{GRepresenter1}, we provide the following example in the Hardy space $H^2(D)$ on the unit disk $D$, where the optimizer of the associated probabilistic approximation problem has a closed measure representation form. 

\begin{example}[Measure Representation]\label{ex:hardy}
Consider the Hardy space $H^2(D)$ on the unit disk $D$, which is an RKHS with Szegő kernel $\phi(w)(z) = \frac{1}{1 - \overline{w} z}$ for $w, z \in D$.
Fix $0<r<1$ and let $D_r$ be the open disk in $\mathbb{C}$ centered at the origin of radius $r$.
Let $\mu$ be the uniform probability measure on $D_r$.
First we check Assumption \ref{cond:bddness_representer_thm} in this setting: for any $\mathbb{F}$-measure $\xi$ with $\supp (\xi) \subset \supp (\mu)$, we have
\begin{equation*}
    \int_D \|\phi(w)\|_{H^2(D)}d|\xi|(w)=\int_{D_r} \frac{1}{1-|w|^2}d|\xi|(w)\leq \frac{1}{1-r^2}\|\xi\|<+\infty,
\end{equation*}
where $\|\xi\|$ is the total variation of $\xi$.
Now let $g \in B^2(D)$, the Bergman space consisting of Lebesgue-$L^2$ holomorphic functions on $D$. Since $d\mu(w) = \frac{1}{|D_r|}d\Area(w)$, then we consider the following minimization problem:
    \begin{equation*}
   \underset{f \in H^2(D)}{\inf} \frac{1}{|D_r|} \int_{D_r} |f(w)-g(w)|^2  d\Area(w) + \lambda \Vert f \Vert^2_{H^2(D)},
\end{equation*}
where $\lambda>0$ and $|D_r| = \pi r^2$.
If we use the power series expressions $f=\sum_{n=0}^\infty a_n z^n$ and $g=\sum_{n=0}^\infty b_nz^n$, then we can apply variations on the coefficients $a_n$ to get an Euler-Lagrange equation for the minimizer $\hat{f}$:
\begin{equation}\label{E-L_eq_of_ex:hardy}
    \left(1+\frac{\lambda(n+1)}{r^{2n}}\right)a_n=b_n.
\end{equation}
We verify that such $\hat{f}=\sum_{n=0}^\infty a_n z^n$ defined by the equation \eqref{E-L_eq_of_ex:hardy} is in $H^2(D)$ since the sequence $\{a_n\}$ is square-summable:
\begin{equation*}
    \sum_{n=0}^\infty |a_n|^2 = \sum_{n=0}^\infty \frac{r^{4n}|b_n|^2}{(r^{2n}+\lambda(n+1))^2} \leq \frac{1}{\lambda^2} \sum_{n=0}^\infty \frac{|b_n|^2}{(n+1)^2} \leq \frac{1}{\lambda^2} \sum_{n=0}^\infty \frac{|b_n|^2}{n+1} <+\infty,
\end{equation*}
where the last inequality follows from that $g=\sum_{n=0}^\infty b_nz^n \in B^2(D)$ with $\|g\|_{B^2(D)}^2 =\sum_{n=0}^\infty \frac{|b_n|^2}{n+1} <\infty$.
Thus, $\hat{f}$ is determined by $g$ via this formula, and our goal is to find a $\mathbb{C}$-measure representation of $\hat{f}$, as in Conjecture \ref{conj:representer}.
\par
To get the above equation \eqref{E-L_eq_of_ex:hardy}, first note that $\|f\|^2_{H^2(D)} = \sum_{n=0}^\infty |a_n|^2$.
For the squared error term, we have 
\begin{equation*}
\begin{split}
    \int_{D_r} |f(w) - g(w)|^2 \, d\Area(w) &= \int_{D_r} \left|\sum_{n=0}^\infty (a_n - b_n) w^n\right|^2 \, d\Area(w) \\
&= \sum_{k=0}^\infty \sum_{n=0}^\infty (a_k - b_k) \overline{(a_n - b_n)} \int_{D_r} w^k \bar{w}^n \, d\Area(w),
\end{split}
\end{equation*}
where the interchange of summation and integration follows from that the series $ \sum_{k=0}^\infty \sum_{n=0}^\infty (a_k - b_k)\overline{(a_n - b_n)}\, w^k \bar{w}^n$ is uniformly convergent in $w \in D_r$.
In polar coordinates, we can write $w = \rho e^{i\theta}$ and $d\Area(w) = \rho d\rho d\theta$.
Then
$$
\int_{D_r} w^k \bar{w}^n \, d\Area(w)
= \int_0^r \int_0^{2\pi} \rho^{k+n} e^{i(k-n)\theta} \rho \, d\theta \, d\rho
= \int_0^r \rho^{k+n+1} d\rho \int_0^{2\pi} e^{i(k-n)\theta} d\theta.
$$
Since
$$
\int_0^{2\pi} e^{i(k-n)\theta} d\theta =
\begin{cases}
2\pi & \text{if } k = n, \\
0 & \text{if } k \neq n,
\end{cases}
$$
we have
\begin{equation}\label{eqn:hardyspace}
    \int_{D_r} w^k \bar{w}^n \, d\Area(w) = 
\begin{cases}
\displaystyle 2\pi \int_0^r \rho^{2n+1} \, d\rho = \frac{\pi r^{2n+2}}{n+1} & \text{if } k = n, \\
0 & \text{if } k \neq n.
\end{cases}
\end{equation}
Therefore, 
$$
\int_{D_r} |f(w) - g(w)|^2 \, d\Area(w) = \sum_{n=0}^\infty |a_n - b_n|^2 \  \frac{\pi r^{2n+2}}{n+1}.
$$
Now we can write the error functional as
$$
J := \frac{1}{|D_r|}\int_{D_r} |f-g|^2  d\Area + \lambda \Vert f \Vert^2_{H^2(D)} = \sum_{n=0}^\infty \frac{ r^{2n}}{n+1} |a_n - b_n|^2+ \lambda \sum_{n=0}^\infty |a_n|^2.
$$
The Euler-Lagrange equation of $J$ over the sequence $\{a_n\}$ can now be derived by setting the variations to zero for each $k$:
\begin{align*}
    0&=\frac{d}{dt}|_{t=0} J(\{a_n\}+te_k)=2\frac{r^{2k}}{k+1}\Re(a_k-b_k)+2\lambda \Re (a_k),\\
    0&=\frac{d}{dt}|_{t=0} J(\{a_n\}+ite_k)=2\frac{r^{2k}}{k+1}\Im(a_k-b_k)+2\lambda \Im (a_k),
\end{align*}
where $e_k$ is the sequence with the $k$-th entry being $1$ and others $0$.
Therefore,
\begin{equation*}
    \frac{r^{2k}}{k+1}(a_k-b_k)+\lambda a_k=0
\end{equation*}
for each $k$, concluding the equation \eqref{E-L_eq_of_ex:hardy}.

\par
To find a $\mathbb{C}$-measure representation of $\hat{f}$, our strategy is to first find the measure representation for the basis vectors $\{z^k\}$ of $H^2(D)$.
For $z, w \in D$, we can write the Szegő kernel as $\phi(w)(z) =\frac{1}{1 - \bar{w} z} = \sum\limits_{n=0}^\infty \bar{w}^n z^n$.
Then computation gives
\begin{equation*}
    \begin{split}
        \int_{D_r} \frac{1}{1 - \bar{w} z} \frac{k+1}{\pi} r^{-2k-2} w^k \, d\Area(w)
&= \frac{k+1}{\pi} r^{-2k-2} \left(\sum_{n=0}^\infty z^n \int_{D_r} w^k \bar{w}^n \, d\Area(w)\right)\\
&= \frac{k+1}{\pi} r^{-2k-2} \left(z^k \ \frac{\pi r^{2k+2}}{k+1}\right) = z^k,
    \end{split}
\end{equation*}
where the first equality is due to that for fixed $z \in D_r$, the series $\sum_{n=0}^\infty z^n w^k \bar{w}^n$ is uniformly convergent in $w \in D_r$, 
and the second equality follows from equation \eqref{eqn:hardyspace}.
We thus see that $z^k$ can be represented by the $\mathbb{C}$-measure $\xi_k$. That is to say,  
$$
z^k = \int_{D_r} \phi(w)(z)d\xi_k(w),
$$
where
\begin{equation*}
    d\xi_k(w):=\frac{k+1}{\pi} r^{-2k-2} w^k d\Area\llcorner D_r(w).
\end{equation*}

From $\hat{f}=\sum_{n=0}^\infty a_nz^n$, we would imagine that $\hat{f}$ is represented by the measure $\nu:=\sum_{n=0}^\infty a_n \xi_n$.
Indeed this is a well-defined $\mathbb{C}$-measure since we have
\begin{align*}
        \|\nu\|\leq \sum_{n=0}^\infty |a_n|  \|\xi_n\|
    &= \sum_{n=0}^\infty \frac{r^{2n}}{r^{2n}+\lambda(n+1)}|b_n| \cdot \frac{2n+2}{n+2}r^{-n} \\
    &\leq \frac{1}{\lambda}\sqrt{\sum_{n=0}^\infty \frac{|b_n|^2}{n+1}}\sqrt{\sum_{n=0}^\infty\frac{4}{n+1}r^{2n}}<+\infty,
\end{align*}
where the first square root is finite since $g=\sum_{n=0}^\infty b_nz^n \in B^2(D)$. 
\par
Now we show precisely that $\hat{f}$ is represented by $\nu$.
Let $\nu_k  = \sum_{n=0}^k a_n \xi_n$ be the partial sums of $\nu$.
Then we have $\hat{f} = \lim_{k \rightarrow \infty} \int_D \phi d \nu_k$ strongly in $H^2(D)$ since $\{z^k\}_{k\in \N_0}$ is an orthonormal basis.
On the other hand, 
$\int_D \phi d \nu_k$ converges weakly to $\int_D \phi d \nu$ in $H^2(D)$ since all functions in $H^2(D)$ are bounded on $D_r$. Thus, by the uniqueness of the weak limit, we have $\hat{f}=\int_D \phi d\nu$.
\par
We also point out that the unregularized problem when $\lambda = 0$ might not have a solution, since the $H^2(D)$-norm of the optimizer can blow up.
Indeed, when $\lambda = 0$ we have $a_n=b_n$ and thus $\| \hat{f}\|_{H^2(D)}^2 = \sum_{n=1}^\infty |b_n|^2$, which can be infinite since $g$ might not be in $H^2(D)$.
Therefore, regularization makes the problem solvable.
\end{example}


\section{Conclusion and Future Work}\label{sec:conclusion}

In this work, we generalize the least squares approximation problem to a probabilistic framework in reproducing kernel Hilbert spaces (RKHSs). By replacing square sum with integration against a probability measure $\mu$, we establish the existence and uniqueness of optimizers under broad conditions (Theorems~\ref{pexistance}, \ref{muPenaltyThm}, and Proposition \ref{constrain:muPenaltyThm}). Our main contribution is that we show representer-type theorems for optimizers (Theorem \ref{GRepresenter1}, Corollaries \ref{GRepresenter2}, \ref{constraint: GRepresenter1}, and \ref{constraint:GRepresenter2}),
which allows us to:

\begin{itemize}
    \item Unify discrete and continuous settings: when $\mu$ is finitely supported, our results recover the classical representer theorem; for general $\mu$, optimizers admit approximations via kernel embeddings. 
    \item Connect the approximation and sampling theory: the finite-dimensional case builds a connection to measure quantization theory, while the Hardy space example demonstrates Conjecture \ref{conj:representer} that has potential applications.  
\end{itemize}

As Conjecture \ref{conj:representer} stated, a central open question in this work is whether the optimizer $\hat{f}$ in Theorem~\ref{GRepresenter1} \emph{always} admits a representation $\hat{f} = \int_X \phi(x) \, d\nu(x)$ for some $\mathbb{F}$-measure $\nu$, even when $\mathscr{H}(K)$ is infinite-dimensional and $\mu$ is not finitely supported. The success of Example~\ref{ex:hardy} on the Hardy space motivates further study of probabilistic approximation in other complex RKHSs (e.g., Bergman or Dirichlet spaces), which remains to be explored in the future. 

\vspace{0.5cm}

 \vspace{0.3cm}

{\bf Data Availability}: No datasets were used in this study.

 \vspace{0.3cm}

{\bf Author Contributions}: All authors contributed equally to the manuscript.

 \vspace{0.3cm}

{\bf Acknowledgment}: This work was partially inspired by probabilistic frames. The authors thank Saburou Saitoh and Qiyu Sun for valuable discussions. The authors also thank the editor and reviewer for their patience and efforts in handling this manuscript. Their feedback has significantly improved the quality of this paper.  
Part of the research in this paper was done when the second author visited NCTS. He would like to thank NCTS for their accommodation and hospitality.

\section*{Declarations}
{\bf Competing Interests}: The authors declare no competing interests.

\bibliographystyle{plain} 
\bibliography{refs}

\end{document}